\documentclass[12pt, reqno]{amsart}

\usepackage{graphicx}
\usepackage{subfigure}
\usepackage{eepic}
\usepackage{xcolor}
\selectcolormodel{gray}
\usepackage{tikz}
\usepackage{amsmath}
\usepackage{a4wide}
\usepackage[utf8]{inputenc}
\usepackage{amssymb}
\usepackage{amsopn}
\usepackage{epsfig}
\usepackage{amsfonts}
\usepackage{latexsym}
\usepackage{listings}
\usepackage{amsthm}
\usepackage{enumerate}
\usepackage[UKenglish]{babel}
\usepackage{verbatim}
\usepackage{color}
\usepackage{bbold}

\newtheorem{thm}{Theorem}[section]
\newtheorem{lma}[thm]{Lemma}

\newtheorem{cor}[thm]{Corollary}
\newtheorem{defn}[thm]{Definition}

\newtheorem{prop}[thm]{Proposition}

\newtheorem{rem}[thm]{Remark}

\newtheorem{exam}[thm]{Example}

\newcommand{\R}{\mathbb{R}}
\newcommand{\tr}{\textnormal{tr}}

\newcommand{\N}{\mathbb{N}}

\newcommand{\C}{\mathbb{C}}

\providecommand{\norm}[1]{\lVert#1\rVert}

\renewcommand{\geq}{\geqslant}
\renewcommand{\leq}{\leqslant}
\renewcommand{\epsilon}{\varepsilon}

\renewcommand{\l}{\text{loc}}
\renewcommand{\i}{\mathbf{i}}

\newcommand{\id}{\textnormal{Id}}

\renewcommand{\geq}{\geqslant}
\renewcommand{\leq}{\leqslant}

\renewcommand{\i}{\mathbf{i}}
\providecommand{\I}{\mathcal{I}}

\renewcommand{\l}{\mathcal{L}}

\providecommand{\p}{\mathbf{p}}
\providecommand{\q}{\mathbf{q}}
\providecommand{\i}{\mathbf{i}}

\providecommand{\norm}[1]{\lVert#1\rVert}

\providecommand{\one}{\mathbf{1}}

\begin{document}

\title[]{Approximating integrals with respect to stationary probability measures of iterated function systems}

\author{Italo Cipriano} \address{Facultad de Matem\'aticas,
Pontificia Universidad Cat\'olica de Chile (PUC), Avenida Vicu\~na Mackenna 4860, Santiago, Chile}
\email{icipriano@gmail.com }
\urladdr{http://www.icipriano.org/}

\author{Natalia Jurga} \address{Department  of  Mathematics,  University  of  Surrey,  Guildford,  GU27XH, UK}
\email{N.Jurga@surrey.ac.uk}
\urladdr{}

\date{\today}

\subjclass[2010]{}

\begin{abstract}
We study fast approximation of integrals with respect to stationary probability measures associated to iterated functions systems on the unit interval. We provide an algorithm for approximating the integrals under certain conditions on the iterated function system and on the function that is being integrated. We apply this technique to estimate Hausdorff moments, Wasserstein distances and Lyapunov exponents of stationary probability measures.
\end{abstract}

\maketitle

\section{Introduction}\label{intro}

Let $\I:=\{1, \ldots, N\}$ for some $N\in\{2,3,\ldots\}$ and let $\Phi=\{\phi_i\}_{i \in \I}$ be an iterated function system consisting of Lipschitz contractions  $\phi_i:[0,1] \to [0,1]$. It is well-known that there exists a unique, non-empty and compact set $\Lambda\subset [0,1]$  such that
$$
\Lambda=\bigcup_{i\in \I} \phi_i(\Lambda).
$$
This set is called the attractor of $\Phi.$ Given a probability vector $\p=(p_1,p_2,\ldots,p_N)$ (meaning that each $0< p_i < 1$  and $\sum_{i\in\I} p_i=1$), there exists a unique probability measure $\mu=\mu^{(\Phi,\p)}$ such that
\begin{equation}\label{SPMeq}
\int \varphi \textup{d}\mu= \sum_{i\in\I} p_i \int  \varphi \circ \phi_i \textup{d}\mu
\end{equation}
for every continuous function $\varphi:[0,1]\to \R.$ This probability measure is called the stationary probability measure associated to $\Phi$ and $\p.$ Its support is the attractor of $\Phi.$

Iterated function systems were first studied by Hutchinson  \cite{Hutchinson} who proved the existence and uniqueness of the attractor and stationary probability measure. Since the late eighties, the problem of algorithmically estimating stationary probability measures has been a topic of active research, see \cite{Strichartz??,Bla01,Fro99,Obe05,Galatolo_2016}.

In these notes we focus on the related problem of algorithmically estimating the \emph{integrals} of functions with respect to stationary measures. In many areas of analysis, dynamical systems and probability theory, important quantities can be expressed in terms of such integrals, see section \ref{app}. For example, in fractal geometry, the Hausdorff dimension of the attractor $\Lambda$ can be bounded or sometimes explicitly calculated in terms of an expression involving the Lyapunov exponent of particular stationary measures, see section \ref{lyap}. On the other hand, in dynamical systems, equilibrium measures often take the form of a stationary probability measure, and therefore it is a natural problem to study the approximation of integrals with respect to them. Apart from some special cases, for instance whenever the stationary measure is absolutely continuous with an explicit known density, it is impossible to analytically calculate these integrals, and so one must turn to approximating them. We use an operator theoretic approach similar to \cite[Theorem 1]{pollicott}, where Jenkinson and Pollicott constructed an algorithm yielding accelerated approximations of integrals of functions with respect to Lebesgue measure. Indeed, our main Theorem \ref{main} can be seen as an analogue of their result. In particular we provide

\begin{itemize}
\item a sufficient condition on the iterated function system $\Phi,$ and
\item a class $\mathcal{C}_{\Phi}$ of functions $g:[0,1]\to \R,$
\end{itemize}
such that the integral 
$$\int_0^1 g \textup{d}\mu$$
can be approximated efficiently, where $\mu$ denotes any stationary probability measure associated to $\Phi$. In particular, in Theorem \ref{main_theorem} we provide an algorithm which produces approximations $\mu_k(g)$ to $\int g\textup{d}\mu$ such that the approximations approach the integral at a super-exponential rate. We demonstrate the performance of this algorithm by using it to estimate Hausdorff moments, Wasserstein distances and Lyapunov exponents.

\section{Main results}

We introduce the following contraction condition motivated by the definition introduced in \cite[Definition 2.2]{bj}.

\begin{defn}
A family of maps $\Phi=\{\phi_i\}_{i \in \I}$, $\phi_i:[0,1] \to [0,1]$ is called \emph{complex contracting} if there exists a non-empty, bounded, connected and open set $\mathcal{D} \subset \C$ such that $[0,1] \subset \mathcal{D}$ and 
\begin{enumerate}
\item each $\phi_i$ extends holomorphically to $\mathcal{D}$
\item $\phi_i'$ is continuous on the boundary of $\mathcal{D}$
\item $\sup_{i \in \I} \sup_{z \in \mathcal{D}} |\phi_i'(z)| <1.$
\end{enumerate}
\end{defn}

The significance of the above definition is that if $\Phi$ is a complex contracting family of maps then there exists some complex domain $D \subset \mathcal{D}$ such that $[0,1] \subset D$, each $\phi_i$ extends holomorphically to $D$ and $\overline{\phi_i(D)} \subset D$, see \cite[Lemma 2.4]{bj}. In this case we say that $D$ is an \emph{admissable domain} for $\Phi$. The existence of such a domain is what is actually required in order to construct our algorithm, however in many cases it is easier to check that a family is complex contracting than to check that such a domain $D$ exists.

\begin{defn}
Given a complex contracting iterated function system $\Phi=\{\phi_i\}_{i \in \I}$ on the unit interval, we define the class of functions  
$$
\mathcal{C}_{\Phi}:=\bigcup_{D}\left\{ g:[0,1] \to \R: g  \mbox{ has an extension to some bounded holomorphic function on } D \right\},
$$
where the union is taken over all admissable domains for $\Phi.$
\end{defn}

In particular if $\Phi$ is complex contracting, then all sufficiently small Euclidean $\epsilon$-neighbourhoods of $[0,1]$ are admissable, see \cite[Remark 2.5]{bj}. Therefore any function which is bounded and holomorphic on some Euclidean $\epsilon$-neighbourhood of $[0,1]$ belongs to $\mathcal{C}_{\Phi}$. 


We introduce some notation. Let $\I^n=\{i_1\cdots i_n: i_j \in \I\}$ and $\phi_{i_1 \ldots i_n}:=\phi_{i_1} \circ \cdots \circ \phi_{i_n}$ for $n\in\N.$ Define $\I^{\ast}= \bigcup_{n \in \N} \I^n$. For $\i=i_1 \ldots i_n \in \I^{\ast}$ and $1 \leq k \leq n-1$ we define
$$
\begin{aligned}
\sigma^k \i:=&i_{k+1} \ldots i_n i_1 \ldots i_k,\\
\tilde{\sigma}^k \i:=&i_{k+1} \ldots i_n \mbox{ and }\\
p_{\i}:=&\prod_{i=1}^n p_i.
\end{aligned}
$$
Define $\Pi: \I^{\N}\to \Lambda$ by 
$$\Pi(i_1 i_2 \ldots):=\lim_{n\to\infty} \phi_{i_1} \circ \cdots \circ \phi_{i_n}\left([0,1]\right).$$
For each $\i \in \I^{\ast}$ let $(\i)^{\infty}$ denote the periodic point $(\i)^{\infty}=\i\i\i... \in \I^{\N}$ and let $z_{\i}$ denote the fixed point $z_{\i}=\Pi((\i)^{\infty})$ of $\phi_{\i}$. Finally, we say that an iterated function system is non-overlapping if $\{\phi_i((0,1))\}_{i \in \I}$ are pairwise disjoint.\\

The following is our main result. 

\begin{thm}\label{main_theorem}
Let $\Phi=\{\phi_i\}_{i \in \I}$ be a complex contracting iterated function system on the unit interval and $\p=(p_1, \ldots, p_N)$ a probability vector. Given $g\in \mathcal{C}_{\Phi},$ define
$$t_m:=\sum_{\i \in \I^m} p_{\i} \frac{1}{1-\phi_{\i}'(z_{\i})}$$
and
$$\tau_m:=\sum_{\i \in \I^m} p_{\i} \frac{g(z_{\i})+ g(z_{\sigma \i})+ \ldots + g(z_{\sigma^m \i}) }{1-\phi_{\i}'(z_{\i})}.$$
Also define
$$\alpha_{n}:= \sum_{l=1}^n \frac{(-1)^l}{l!} \sum_{n_1+ \ldots +n_l=n} \sum_{j=1}^l \frac{\tau_{n_j}}{n_j} \prod_{1 \leq m \leq l, m \neq j} \frac{t_{n_m}}{n_m}$$
and
$$a_n:= \sum_{l=1}^n \frac{(-1)^l}{l!} \sum_{n_1+ \ldots +n_l=n} \prod_{i=1}^l \frac{t_{n_i}}{n_i}.$$
Then for
$$\mu_k(g):= \frac{\sum_{n=0}^k \alpha_n}{\sum_{n=0}^k na_n}$$
and $\mu$ the stationary probability measure associated to $\Phi$ and $\p,$  
\begin{eqnarray}
\left|\int_0^1 g \textup{d} \mu-\mu_k(g)\right| < C\exp(-\lambda k^2)
\label{error}
\end{eqnarray} for some constants $C, \lambda>0$ which are independent of $k.$
\label{main}
\end{thm}

Note that while our approximations are non-effective, in some special cases the constants $C, \lambda$ appearing in (\ref{error}) can be made explicit. For example, when each $\phi_i$ takes the form of a linear fractional transformation, that is, $\phi_i(x)= \frac{a_ix+b_i}{c_ix+d_i}$ for some constants $a_i, b_i, c_i, d_i$ (which in particular includes the important class of self-similar iterated function systems), the constants $C$ and $\lambda$ can be bounded similarly to \cite{mj, pj}. Also note that since the time taken to process $n$ steps of the algorithm is exponential in $n$, the error decreases super-polynomially fast in time. 

Finally, we observe that by using stationarity of the measure $\mu$, Theorem \ref{main} can also be used to obtain approximations of integrals of some \emph{piecewise analytic} functions. In particular, suppose that for some $K \in \mathbb{N}$, $\{\phi_{\i}([0,1])\}_{\i \in \I^K}$ are pairwise disjoint and let $g:[0,1]\to\R$ be a function (not necessarily continuous) such that each composition function $g \circ \phi_{\i}: [0,1]\to\R$ belongs to $\mathcal C_{\Phi}$ for every $\i \in \I^K.$ By continuous continuation there is a continuous function $\tilde{g}:[0,1]\to\R$ with $\tilde{g}|_{\phi_{\i}([0,1])}=g|_{\phi_{\i}([0,1])}$ for every $\i\in \I^K.$ Thus by stationarity,
\begin{eqnarray*}
\int_{0}^1 g(x) \textup{d}\mu (x)&=& \int_{0}^1 \tilde{g}(x) \textup{d}\mu (x) \\
&=& \sum_{\i \in \I^K} p_{\i} \int_{0}^1 \tilde{g} \circ \phi_{\i} (x) \textup{d}\mu (x) \\
&=&  \sum_{\i \in \I^K} p_{\i} \int_{0}^1 g \circ \phi_{\i}(x) \textup{d}\mu(x),
\end{eqnarray*}
so since $\phi_{\i}(D) \subset D$, Theorem \ref{main} can be applied to approximate each integral $ \int g \circ \phi_{\i} \textup{d}\mu$.


\section{Preliminaries}

\subsection{Trace class operators, determinants and approximation numbers} Given a compact operator $L:H \to H$ on a Hilbert space $H$, its $n$th \emph{approximation number} is defined as
$$s_n(L)= \inf\{\norm{L-K} : \textnormal{rank}(K) \leq n-1\}.$$

A bounded linear operator on a complex separable Hilbert space $H$ is called \emph{trace-class} if $\sum_{n=1}^{\infty} s_n(L)< \infty.$ Given a trace-class operator $L$, the \emph{trace} is defined as
$$\tr (L)= \sum_{n=1}^{\infty} \langle Le_n, e_n \rangle_H$$
where $\{e_n\}$ is any orthonormal basis and $\langle , \rangle_H$ is the inner product for the Hilbert space $H$. 

Given a compact operator $L$, we denote by $\{\lambda_n(L)\}_{n \in \N}$ the monotone decreasing sequence of non-zero eigenvalues of $L$, listed with algebraic multiplicity. If $L$ is trace-class then it is compact and its sequence of eigenvalues $\lambda_n(L)$ is absolutely summable. 

For a trace-class operator $L$, the \emph{Fredholm determinant} of $L$ can be defined as
\begin{eqnarray}
\det(\id -zL)= \prod_{n=0}^{\infty} (1-z\lambda_n(L))
\label{det}
\end{eqnarray}
which is an entire function of $z$ \cite[Theorem 3.3]{bs}, so in particular there exist $a_n \in \C$ such that
$$\det(\id-zL)= \sum_{n=0}^{\infty} a_nz^n.$$
Note that by (\ref{det}) the roots of $\det(\id -zL)$ are precisely the reciprocals of the eigenvalues of $L$, and the degree of each zero is given by the multiplicity of the corresponding eigenvalue. Moreover, each coefficient $a_n$ can be expressed in terms of the traces of $L^m$ for $1 \leq m \leq n$:
\begin{eqnarray}a_n= \sum_{m=1}^n \frac{(-1)^m}{m!} \sum_{\substack{n_1, \ldots, n_m \in \N^m \\ n_1 + \ldots +n_m=m}} \prod_{i=1}^m \frac{\textnormal{tr} L^{n_i}}{n_i}, \label{tc1} \end{eqnarray}
for a proof see for instance \cite[Proposition 3.2]{mj}.

On the other hand, by finding the coefficient of $z^n$ in (\ref{det}) we see that 
$$a_n=(-1)^n \sum_{i_1< \ldots <i_n} \lambda_{i_1}(L) \ldots \lambda_{i_n}(L)$$
therefore
\begin{eqnarray}
|a_n| \leq \sum_{i_1< \ldots< i_n} |\lambda_{i_1}(L) \ldots \lambda_{i_n}(L)| .\label{sing coef}
\end{eqnarray}

\subsection{Bergman space} \label{bergman}

We will use results of Bandtlow and Jenkinson \cite{bj, bj2} on operators acting on the Bergman space. Let $D$ be any connected, non-empty open subset of $\mathbb{C}$. We define the Bergman space $A^2(D)$ by
$$A^2(D):= \left\{ f:D \to \C \quad \textnormal{holomorphic} \quad : \quad \int_D |f(z)|^2 \textup{d}V(z)< \infty\right\}$$
where $V$ denotes 2-dimensional Lebesgue measure, normalised so that the unit ball has unit mass. $A^2(D)$ is a Hilbert space with inner product
$$\langle f,g\rangle_{A^2}= \int_D f(z)\overline{g(z)} \textup{d}V(z).$$

Consider the operator
$$Lf(z)= \sum_{i=1}^k w_i(z)f(\psi_i(z))$$
where $\psi_i$ are holomorphic functions on $D$ that satisfy $\bigcup_{i=1}^k \overline{\psi_i(D)} \subset D$ and $w_i$ are holomorphic bounded functions on $D$. Using the work of Bandtlow and Jenkinson \cite{bj,bj2} we have the following result.

\begin{prop} \label{bj}
The operator $L$ preserves the Bergman space $A^2(D)$. Moreover
\begin{enumerate}
\item  $L$ is a trace-class operator and there exist  constants $C, \lambda>0$ which are independent of $N$ and $w_i$ such that its eigenvalues satisfy
$$\lambda_n(L) \leq C\left(\sum_{i=1}^k \norm{w_i}_{\infty}\right) \exp(-\lambda n)$$
where $\norm{w_i}_{\infty}:= \sup_{z \in D} |w_i(z)|$,
\item the trace of $L^n$ is given by
$$\tr(L^n)= \sum_{\i \in \I^n} \frac{w_\i(z_\i)}{1-\psi_\i'(z_\i)}$$
where $z_{\i}$ is the unique fixed point of $\psi_\i$ and $w_{\i}(z) := w_{i_n }(z) w_{i_{n-1} }(\phi_{i_n}z)\cdots w_{i_1}( \phi_{i_2 \ldots i_n} z ).$
\end{enumerate}
\end{prop}

\begin{proof}
The first part is a special case of \cite[Theorems 5.9 and 5.13]{bj2}. The second part follows by \cite[Theorem 4.2]{bj}.
\end{proof}

\subsection{Analytic perturbation theory}

We say that a bounded linear operator $L$ on a Banach space has \emph{spectral gap} if $L=\lambda P+N$ where $P$ is a rank one projection (so $P^2=P$ and $\dim(\textnormal{Im}(P))=1$), $N$ is a bounded operator with spectral radius $\rho(N)< |\lambda|$ and $PN=NP=0$. $L$ does not need to be compact in order to have a spectral gap, however if the operator $L$ \emph{is} compact and has a simple leading eigenvalue\footnote{Throughout the paper we say that an eigenvalue is simple if it is \emph{algebraically simple}, that is, the eigenvalue has a one-dimensional generalised eigenspace.} and no other eigenvalues with the same absolute value, it has a spectral gap. 

We can use the standard techniques of perturbation theory \cite{kato} to relate $\int g\textup{d}\mu$ to the spectral properties of an appropriate operator. The following perturbation theorem is presented in a more general form in \cite[Theorem 3.8]{hennion}.

\begin{thm}[Analytic perturbation theorem] \label{apt} Let $\{L_t\}_{t \in \C}$ be a family of bounded linear operators on a Banach space such that $t \mapsto L_t$ is holomorphic and $L_0$ has spectral gap. Then there exists an open neighbourhood $U \subset \mathbb{C}$ of 0 for which $L_t$ has spectral gap for all $t \in U$. Moreover  there exist $\lambda(t), P_t, N_t$ which are holomorphic families on $U$ such that:
\begin{enumerate}
\item[(a)] $L_t=\lambda(t) P_t+N_t$,
\item[(b)] $N_t P_t=P_t N_t=0$
\item[(c)] $P_t$ is a bounded rank one projection and has the form $$P_t= \frac{1}{2\pi i} \int_{\gamma} (s \id -L_t)^{-1} \textup{d}s $$
for some small circle $\gamma$ around $\lambda$ which separates it from the rest of the spectrum of $L_0$,
\item[(d)] $\rho(N_t)< |\lambda(t)|- \epsilon$ for some $\epsilon>0$ which is independent of $t$.
\end{enumerate}
\end{thm}

\section{Transfer operator}

Let $\Phi$ and $g$ satisfy the hypothesis of theorem \ref{main}. By assumption, there exists an admissable domain $D$ for $\Phi$ such that $g$ has a bounded and holomorphic extension to $D$. Abusing notation slightly, we also denote this extension by $g$. Fix a probability vector $\p$, $s \in \C$ and $f\in A^2(D)$. Define the operator $\l_s$ as
$$\l_sf(z)= \sum_{i=1}^N p_i \exp(sg(\phi_i (z)))f(\phi_i(z)).$$
Then by Proposition \ref{bj}, $\l_s f \in A^2(D)$. 
Observe that the iterates of $\l_s$ are given by
$$\l_s^nf(z)= \sum_{\i \in \I^n} p_{\i} \exp(s(g(\phi_{\i} z)+g(\phi_{\tilde{\sigma} \i}z) + \ldots + g(\phi_{\tilde{\sigma}^{n-1}\i} z)))f(\phi_{\i}z).$$

The following proposition summarises some of the spectral properties of $\l_s$.

\begin{prop} \label{spectral}
$\l_s:A^2(D) \to A^2(D)$ is a trace-class operator with decreasing sequence of eigenvalues $\{\lambda_n(s)\}_{n \in \N}$.  Moreover
\begin{enumerate}
\item There exist constants $C, \lambda>0$ which are independent of $s$ such that the eigenvalues admit the bound
$$\lambda_n(\l_s) \leq C\exp(-\lambda n^2)$$
for any $s \in \C$ with $|s| \leq 1$,
\item  $\lambda_1(0)=1$ is a simple eigenvalue of $\l_0$ of maximum modulus,
\item $\lambda_1(s)$ is analytic in $s$ in a neighbourhood of 0 and
\item $\int g \textup{d} \mu= \frac{\textup{d}}{\textup{d}s}\lambda_1(s)\biggr|_{s=0}$.
\end{enumerate}
\end{prop}

\begin{proof}
The fact that $\l_s$ is trace-class and the uniform bounds on $\lambda_n(\l_s)$ follow from Proposition \ref{bj}. 

Next we prove (2). Clearly 1 is an eigenvalue of $\l_0$ with eigenfunction $\one$. To see that it is geometrically  simple, suppose that $f \in A^2(D)$ is a fixed point of $\l_0$ and $f \neq 0$. We will show that $f$ must be a constant function. First, observe that
$$|f(z)|=|\l_0 f(z)| \leq \sum_{i=1}^k p_i|f\circ \phi_{i}(z)| \leq \sup_{z' \in \bigcup_{i=1}^k \overline{\phi_{i}(D)}} |f(z')|$$
where the right hand side is finite because $\bigcup_{i=1}^k \overline{\phi_{i}(D)}$ is a compact subset of $D$. Therefore, 
$$\sup_{z \in D}|f(z)| \leq\sup_{z' \in \bigcup_{i=1}^k \overline{\phi_{i}(D)}} |f(z')|=|f(z_0)|$$
for some $z_0 \in \bigcup_{i=1}^k\overline{\phi_{i}(D)} $. By the maximum-modulus principle, $f$ is constant on $D$. Using the same argument, we can establish that 1 is the only eigenvalue of modulus 1.

Therefore it remains to show that $1$ is an algebraically simple eigenvalue. We need to show that $\ker(\l_0-\id)^2$ is one dimensional (so only consists of the constant functions). For a contradiction suppose that there exists $f \in A^2(D)$ for which $(\l_0-\id)f \neq 0$ but $(\l_0-\id)f \in \ker(\l_0-\id)$. So in particular $(\l_0-\id)f = c \one$ for some constant $c$. In particular, $c \neq 0$ since $(\l_0-\id)f \neq 0$ and therefore by replacing $f$ by $c^{-1} f$ we obtain that $(\l_0-\id)f=\one$, that is, $\l_0f= \one +f$. By induction we see that 
\begin{eqnarray}\l_0^n f= n \one+f. \label{ind} \end{eqnarray}

On the other hand, define 
$$\Gamma^n= \overline{ \bigcup_{\i \in \I^n} \phi_{\i}(D)}$$
and define
$\Gamma= \bigcap_{n=1}^{\infty} \Gamma^n$. Since $\Gamma^n$ is a nested sequence of closed subsets of $D$, $\Gamma$ is a compact subset of $D$. For any $z \in \Gamma$, 
\begin{eqnarray}
|\l_0^nf(z)| &=& \left| \sum_{\i \in \I^n} p_{\i}  f(\phi_{\i}(z))\right| \\
&\leq & \sup_{z \in \Gamma} |f(z)|.
\end{eqnarray}
By (\ref{ind}), $|\l_0^n f(z)|=|n+f(z)| \geq n-|f(z)|$ implying that
$$n \leq 2\sup_{z \in \Gamma} |f(z)|$$
which is clearly a contradiction since $f$ is bounded on $\Gamma$.

To see (3), since $\l_0$ is compact, $\lambda_1(0)$ is a simple eigenvalue of maximum modulus and $\{\l_s\}$ is clearly an analytic family of operators in $s$, we can apply the analytic perturbation theorem \ref{apt} to deduce that $\lambda_1(s)$ is analytic in a neighbourhood $U$ of 0.

Finally to prove (4) put
$$\mathcal{P}_s= \frac{1}{2\pi i} \int_{\gamma} (t \id -\l_s)^{-1} \textup{d}t $$
as in (c) of Theorem \ref{apt}. (a)-(c) of Theorem \ref{apt} imply that the image of $\mathcal{P}_s$ is an eigenspace for the eigenvalue $\lambda_1(s)$ and that $h_s= \mathcal{P}_s \one$ is an eigenfunction for the eigenvalue $\lambda_1(s)$. Note that $h_0= \one$. Since $s \mapsto \mathcal{P}_s$ is holomorphic it immediately follows that $s \mapsto h_s$ is also holomorphic. We write $f_0= \frac{\textup{d}}{\textup{d}s} h_s \bigr|_{s=0} \in A^2(D)$.

Fix some $z_0 \in \Gamma \cap (0,1)$. Observe that for each $n>1$ and $s \in U$,
\begin{eqnarray*}
\lambda_1(s)^n h_s(z_0)&=& (\l_s^n h_s)(z_0)= \sum_{\i \in \I^n} p_{\i} \exp(s( g(\phi_{\i}z_0)+ \ldots g(\phi_{\tilde{\sigma}^n\i}z_0))) h_s(\phi_{\i}(z_0)).
\end{eqnarray*}
Denote $S_ng(\phi_{\i} z_0)=g(\phi_{\i}z_0)+ \ldots g(\phi_{\tilde{\sigma}^n\i}z_0)$. Differentiating at $s=0$ we obtain
\begin{eqnarray*}
n\lambda_1'(0) + f_0(z_0)&=& \sum_{\i \in \I^n} p_{\i} S_ng(\phi_{\i} z_0)+ \sum_{\i \in \I^n} p_{\i} f_0(\phi_{\i}(z_0)) \\
&=& \sum_{\i \in \I^n} p_{\i} S_ng(\phi_{\i} z_0)+ \l_0^n f_0(z_0)
\end{eqnarray*}
where we used that $\lambda_1(0)=1$ and $h_0= \one$. Therefore since $|\l_0^nf_0(z_0)-f_0(z_0)| \leq \sup_{z \in \Gamma} |f_0(z)|< \infty$
\begin{eqnarray*}
\lambda_1'(0)&=& \lim_{n \to \infty} \frac{1}{n} \sum_{\i \in \I^n} p_{\i} S_ng(\phi_{\i} z_0) + \lim_{n \to \infty} \frac{1}{n} (\l_0^nf_0(z_0)-f_0(z_0))\\
&=&\lim_{n \to \infty} \sum_{\i \in \I^n} p_{\i} g(\phi_{\i} z_0)= \int g \textup{d}\mu.
\end{eqnarray*}
The fact that $\lim_{n \to \infty} \frac{1}{n} \sum_{\i \in \I^n} p_{\i} S_ng(\phi_{\i} z_0)=\lim_{n \to \infty} \sum_{\i \in \I^n} p_{\i} g(\phi_{\i} z_0)$ follows because $\phi_i$ are uniformly contracting and $g$ has bounded derivative on $[0,1]$, therefore for any $\epsilon, \delta>0$ one can choose $N$ sufficiently large so that for $n \geq N$, and all $k \leq (1-\epsilon)n$,
\begin{eqnarray*}|g(\phi_{\i}z_0)-g( \phi_{\tilde{\sigma}^k\i} z_0)| \leq \delta.\end{eqnarray*}
\end{proof}

\section{Determinants and the algorithm}

Since $\l_s$ is a trace-class operator for each $s \in \C$, we can define its determinant function $\det(\id-z\l_s)$. The following proposition summarises its properties.

\begin{prop}
For all $s \in \C$, $\det(\id-z\l_s)$ is an entire function of $z$ and can be written in the form
$$\det(\id-z\l_s)= \sum_{n=0}^{\infty} b_n(s)z^n$$
for $b_n(s) \in \C$ where $b_0(s)=1$ for all $s$ and for $n\geq 1$ is defined as
\begin{eqnarray}b_n(s)= \sum_{m=1}^n \frac{(-1)^m}{m!} \sum_{\substack{n_1, \ldots, n_m \in \N^m \\ n_1 + \ldots +n_m=n}} \prod_{i=1}^m \frac{\textnormal{tr} \l_s^{n_i}}{n_i}. \label{tc} \end{eqnarray}
Moreover, the trace of $\l_s^n$ is given by
\begin{eqnarray}
\tr(\l_s^n)= \sum_{\i \in \I^n} p_{\i} \frac{\exp(s(g(z_{\i})+g(z_{\sigma \i}) \ldots +g(z_{\sigma^{n-1}\i})))}{1-\phi_{\i}'(z_{\i})}. \label{tf} \end{eqnarray}
\label{tc prop}
\end{prop}

\begin{proof}
(\ref{tc}) follows from (\ref{tc1}). To see (\ref{tf}) notice that by Proposition \ref{bj}(2),
$$\tr(\l_s^n)= \sum_{\i \in \I^n} p_{\i} \frac{\exp(s(g(\phi_{\i}(z_{\i}))+g(\phi_{\tilde{\sigma}\i}(z_{\i}))+ \ldots +g(\phi_{\tilde{\sigma}^k\i}(z_{\i}))))}{1-\phi_{\i}'(z_{\i})}.$$ 
Since for any $1 \leq k \leq n$, $\phi_{\tilde{\sigma}^k \i}(z_{\i})$ is a fixed point of $\phi_{\sigma^k \i}$ it follows that $\phi_{\tilde{\sigma}^k \i}(z_{\i})=z_{\sigma^k\i}$ and therefore (\ref{tf}) follows.
\end{proof}

\begin{prop}
Let $\det(\id -z\l_s)= \sum_{n=0}^{\infty} b_n(s)z^n$ as before. Then there exist constants $C, \lambda >0$ such that $|b_n(s)|, |b_n'(0)| \leq C\exp(-\lambda n^2)$ for all $s \in \C$, $|s| \leq 1$. Moreover,
\begin{eqnarray}
\int g \textup{d}\mu= \lambda_1'(0)= \frac{\sum_{n=0}^{\infty} b_n'(0)}{\sum_{n=0}^{\infty} nb_n(0)}. \label{main2}
\end{eqnarray} \label{main prop} \end{prop}

\begin{proof}
For the first part, note that by (\ref{sing coef}) and Proposition \ref{spectral}(1),
\begin{eqnarray*}
|b_n(s)| &\leq& \sum_{i_1< \ldots <i_n} \lambda_{i_1}(\l_s) \cdots \lambda_{i_n}(\l_s) \\
&\leq& C^n \sum_{i_1< \ldots <i_n} \exp(-\lambda(i_1+ \ldots +i_n)) \\
&\leq& \frac{C^n \exp(-\frac{n(n+1)}{2}\lambda)}{\prod_{i=1}^n (1-\exp(-i\lambda))}
\end{eqnarray*}
where the last line follows by repeated geometric summation. The upper bound on $|b_n(s)|$ follows (for new constants $C$ and $\lambda$). By the Cauchy integral formula,
$$b_n'(0)= \frac{1}{2\pi i} \int_{|s|=1} \frac{b_n(s)}{s^2} \textup{d}s$$
therefore
$$|b_n'(0)| \leq  \sup_{|s|=1} |b_n(s)|$$
from which we obtain the bound on $|b_n'(0)|$.

Let $U$ be the neighbourhood of 0 on which $\lambda_1(s)$ is analytic. Observe that since the zeroes of the determinant $\det(\id-z\l_s)$ are the reciprocals of the eigenvalues of $\l_s$,
\begin{eqnarray}\sum_{n=0}^{\infty} b_n(s)\lambda_1(s)^{-n}=0.\label{root} \end{eqnarray} 
Since $|b_n(s)|=O(\exp(-\lambda n^2))$ uniformly on $U$ we can apply the Cauchy integral formula to deduce that the partial sums $\sum_{n=1}^N b_n'(s)-nb_n(s)\lambda_1'(s)$ converge uniformly on compact subsets of $U$ as $N \to \infty$. Therefore we can differentiate (\ref{root}) and take derivatives inside the summation to obtain 
\begin{eqnarray}
0= \frac{\textup{d}}{\textup{d}s} \left(\sum_{n=0}^{\infty} b_n(s) \lambda_1(s)^{-n} \right) \biggr|_{s=0}= \sum_{n=0}^{\infty} b_n^{\prime}(0)-nb_n(0)\lambda_1'(0). \label{ting} \end{eqnarray}
Since $(\lambda_1(0))^{-1}$ is a simple zero of $\det(\id-z \l_0)$, it follows that $\sum_{n=0}^{\infty} nb_n(0) \neq 0$ and so by rearranging (\ref{ting}) we obtain
$$\lambda_1'(0)= \frac{\sum_{n=0}^{\infty} b_n'(0)}{\sum_{n=0}^{\infty} nb_n(0)}$$
which completes the proof.
\end{proof}

We define the $k$th approximation
\begin{eqnarray}
\mu_k(g):= \frac{\sum_{n=0}^{k} b_n'(0)}{\sum_{n=0}^k nb_n(0)}.\label{approx}
\end{eqnarray}
Define
$$t_m:=\tr(\l_0^m)=\sum_{\i \in \I^m} p_{\i} \frac{1}{1-\phi_{\i}'(z_{\i})}$$
and
$$\tau_m:= \frac{\textup{d}}{\textup{d}s} \tr(\l_s^m) \biggr|_{s=0} =\sum_{\i \in \I^m} p_{\i} \frac{g(z_{\i})+g(z_{\sigma \i}) \ldots +g(z_{\sigma^{n-1}\i}) } {1-\phi_{\i}'(z_{\i})}.$$
We are now ready to prove the main theorem.

\vspace{2mm}

\noindent \textit{Proof of Theorem \ref{main}.} 
By (\ref{tc}), (\ref{approx}) and the definitions of $t_n, \tau_n$ it follows that $\mu_k(g)$ is given by theorem \ref{main}. By the bounds on $b_n(0)$ and $b_n'(0)$ given in Proposition \ref{main prop} and (\ref{main2}) we obtain the desired bound on the error $|\int g\textup{d}\mu-\mu_k(g)|$. \qed

\section{Applications} \label{app}
In this section we demonstrate Theorem \ref{main_theorem} by estimating moments and Lyapunov exponents of stationary probability measures and the first Wasserstein distance between stationary probability measures. 

\subsection{Moments of stationary probability measures}

The $n$-th moment of a positive Borel measure $\mu$ on $\R$ is defined by
$$
\gamma_n=\gamma_n[\mu]:=\int_{-\infty}^{\infty} x^n \textup{d}\mu(x), n=0,1,\ldots.
$$

Moments are particularly important in analysis, probability and statistics. In analysis, the classical moment problem \cite{Akhiezer} is connected with characterising the image of the map
\begin{eqnarray*}
\mathcal{S}:\{\mbox{ positive Borel measures }\mu \mbox{ on }\R: \mu(\R\setminus [-R,R])=O(R^{-\infty})\}\to \R^{\N}, & & \mu\mapsto (\gamma_n[\mu])_n ,
\end{eqnarray*}
which is connected to the problem of the extension of positive functionals. In probability and statistics, the method of moments \cite{Diaconis_Moments} is useful for proving limit theorems and estimating distributions of samples. 

Since moments of stationary probability measures can be computed analytically in some special cases (see Lemma \ref{lem_moments}), we can demonstrate the efficiency of the algorithm by using it to approximate moments. 

\begin{cor}\label{app_Moments}
Let $\Phi$ and $\p$ satisfy the assumptions of Theorem \ref{main}. Then, for each fixed $m \in \N$, Theorem \ref{main} provides an algorithm with a super-exponential rate of convergence (\ref{error}) that gives approximations $w_k\in\R^{m+1}$ to $\gamma_{[0,m]}:=(\gamma_0,\gamma_1,\ldots,\gamma_m)$  for $k=1,2,\ldots.$ 

\end{cor}

 In the setting where $\Phi$ is a non-overlapping iterated function system of similarities it is possible to obtain the following analytic formula for the moments.

\begin{lma}\label{lem_moments}
Let $\Phi=(\phi_1,\ldots,\phi_N),$ $\phi_i(x)=\rho_i x +t_i$ with $0< \rho_i<1$ such that $\phi_i(0,1)\cap \phi_j(0,1)=\emptyset$ for every $i\neq j$ and $\p= (p_1,\ldots,p_N)$ a probability vector.

Then $\gamma_0=1$ and for every $n>0$
$$
\gamma_n=\frac{\sum_{i=0}^{n-1} \binom{n}{i} \gamma_i  \sum_{j=1}^N p_j \rho_j^i t_j^{n-i} } {1- \sum_{j=1}^N p_j \rho_j^n }.
$$  
\end{lma}

\begin{proof}
Directly from \eqref{SPMeq} and definition of $\Phi.$
\end{proof}

We can use Theorem \ref{main_theorem} to compute approximate values for the moments and compare these with the exact values given by Lemma \ref{lem_moments}.

\begin{exam}
 Let $\Phi=\{\frac{1}{3}x, \frac{1}{3}x+\frac{2}{3}\}$ and $p=\left( \frac{1}{3},\frac{2}{3} \right).$ In table \ref{tab:table1} we compare the approximate values for the moments with the exact values given by the formula.

{\tiny
\begin{table}[h!]
  \begin{center}
    \caption{Approximate values for the moments.}
    \label{tab:table1}
    \begin{tabular}{l | c | c|r} 
      \textbf{Order} & \textbf{Approx. moment value ($k=14$)} & \textbf{Actual moment value} &  \textbf{Error} \\
      $n$ & $$ & $\gamma_n$ &  \\
      \hline
      0 & 1 & 1 & 0\\\
      1 &  $0.\bar{6}$ & 2/3 & 0 \\\
      2 &  $0.555555555555555555555555555555555555555555555554956$ & $5/9$ & $<10^{-48}$\\\
      3 & $0.495726495726495726495726495726495726495726495726296$  & $58/117$ & $<10^{-48}$ \\\
      4 &  $0.455270655270655270655270655270655270655270655272966$ &$ 799/1755$ & $<10^{-47}$\\\
      5 &  $0.424681939833454984970136485288000439515591030742856$ &   $ 54110/127413$  & $<10^{-48}$ \\\
      6 &  $0.400239922384444528966673488818010962533107055240313$ &  $ 662945/1656369$  & $<10^{-46}$\\\
      7 &  $0.380103465368104155158312752269065347061865131565893$ &  $ 2064430846/5431233951$  & $<10^{-46}$ \\\
      8 &  $0.363174113035119080714419595189042981452860757191787$ & $ 1213077397297/3340208879865$  & $<10^{-46}$ \\\ 
      9 &  $0.348713510641572814140979346643070945142831570520821$ &  $ 764170684622650/2191399705783431$  & $<10^{-46}$\\\
     10 & $0.336192206948823127238914429058461576988475565523927$ &  $16313445679660723325/48524163685162512633$  & $<10^{-46}$
    \end{tabular}
  \end{center}
\end{table}
}
\end{exam}

On the other hand, when $\Phi$ is made up of non-affine contractions, there is no analytic formula available. In the following example we approximate the moments of a stationary probability measure in this setting.

\begin{exam}
Let $\Phi=\{\frac{1}{x+2},\frac{1}{x+4}\}$ and $\p= \left(\frac{1}{2} , \frac{1}{2} \right).$ We can compute the approximate values for the moments. The results are in Table \ref{tab:table2}, where all of the digits are stable.\footnote{We say that a digit is stable if, from empirical observations, it appears to have converged to a stable value}
{\tiny
\begin{table}[h!]
  \begin{center}
    \caption{Approximate values for the moments.}
    \label{tab:table2}
    \begin{tabular}{l| l} 
      \textbf{Order} & \textbf{Approximate moment value ($k=13$)} \\
      $n$ & $$  \\
      \hline
      0 & $1$ \\
      1 & $0.330469717526485534080138479518406828981534429410127592033533774023242108\ldots$ \\
      2 & $0.1192803776960544798961200249581823359145663180489550186633549589883397537\ldots$ \\
      3 & $0.0461208401857915310881276274089274103313021776737494744720364563940134891\ldots$ \\
      4 & $0.0186956679319404288549585313723378148471406387052756144239642197703097910\ldots$ \\
      5 & $0.0078078479770635609245553780159130351213591283122475360490644426811585893\ldots$ \\
      6 & $0.0033201105732319037686859365646767147485917239274989027585387616786231639\ldots$ \\
      7 & $0.00142718211837850365241828109336276663400252833572475441921650238220944521\ldots$ \\
      8 & $0.000617598307785531412407227175919292282332872096283836068810742529448255746\ldots$ \\
      9 & $0.000268421862695922651075727295017088906810341162763633320631290756501684091\ldots$ \\
     10 & $0.000117017198360695628842316148053569471108768123919492595293008649132277667\ldots$ 
    \end{tabular}
  \end{center}
\end{table}
}
\end{exam}

\subsection{First Wasserstein distance between stationary probability measures} \label{wsection}

Fraser \cite{Fraser}  initiated the study of the interaction between fractal geometry and optimal transportation  by proposing the problem of computing the Wasserstein distances between stationary probability measures. The Wasserstein distance is a metrization of the weak-$*$ topology of the space of probability Borel measures on a Polish space. For $m\in [1,\infty),$ the Wasserstein distance of order $m$ between two probability measures $\mu$ and $\nu$ is defined by
$$
W_{m}(\mu,\nu)=\inf \left\{ \left[\mathbb{E}d(X,Y)^m\right]^{\frac{1}{m}}: \mbox{law}(X)=\mu, \mbox{law}(Y)= \nu  \right\}.
$$
 In general, under no assumption on the iterated functions system and associated stationary probability measure, there is no hope of finding an analytic formula. Under some very restrictive assumptions it has been proved in \cite{Fraser, Mark_Pollicott_Italo_Cipriano} that $W_1(\mu,\nu)$ can be obtained explicitly. Using results announced in \cite{ItaloWPreprint} and Theorem \ref{main} we can provide an estimate of $W_1(\mu,\nu)$ when $d=1$ and $\mu,\nu$ are stationary probability measures. The following result was obtained in \cite{ItaloWPreprint}. 

\begin{thm}\label{teo_pre-print}
Let $\Phi=\{\phi_i\}_{i\in\I}$ be a non-overlapping iterated function system of differentiable Lipschitz contractions on the unit interval.  Additionally, suppose that $\frac{d\phi_i}{dx}>0$ for each $i \in \I$ and that $(\p,\mathbf{q})$ is a pair of probability vectors with the property that $$\sum_{j=1}^i (p_j - q_j)$$  does not change sign for $i=1,\ldots,N$. Then 

$$
W_1\left(\mu^{(\Phi,\p)},\mu^{(\Phi,\q)} \right)= \left| \int x \textup{d} \mu^{(\Phi,\p)} - \int x \textup{d} \mu^{(\Phi,\q)} \right |. 
$$

\end{thm}

Using the above result and Theorem \ref{main_theorem}, we obtain the following.

\begin{cor}\label{app_Wasserstein_dist}
Let $\Phi$ and $\p,\q$ satisfy the hypothesis of Theorem \ref{teo_pre-print} and Theorem \ref{main}. Then Theorem \ref{main} provides an algorithm that gives approximations $w_k$  $(k=1,2,\ldots)$ to $W_1(\mu^{(\Phi,\p)},\mu^{(\Phi,\q)})$ and which converges at a super-exponential rate (\ref{error}).
\end{cor}

Let $\Phi, \p, \q$ satisfy the assumptions of Corollary \ref{app_Wasserstein_dist} and additionally suppose that $\phi_i(x)=\rho_i x +t_i$ with $0< \rho_i<1$ for each $i=1, \ldots N$. In \cite{ItaloWPreprint} it was proved that
\begin{equation}\label{analytic_Wasserstein_dist}
W_1(\mu^{(\Phi,p)},\mu^{(\Phi,q)})=\left| \frac{\sum_{i=1}^N p_i t_i }{1-\sum_{i=1}^N p_i \rho_i }-  \frac{\sum_{i=1}^N q_i t_i }{1-\sum_{i=1}^N q_i \rho_i }  \right|.
\end{equation}

We can compare the approximate values for the Wasserstein distances given by Theorem \ref{app_Wasserstein_dist} with the exact value given by the equation \eqref{analytic_Wasserstein_dist}.

\begin{exam}
 Let $\Phi=\left\{\frac{1}{3}x, \frac{1}{2}x+\frac{1}{2}\right\}$ and $p=\left( \frac{1}{3},\frac{2}{3} \right),$  $q=\left( \frac{3}{4},\frac{1}{4} \right).$ We compare the approximate value of the Wasserstein distance $w_k$ for $k=8,9,\ldots, 16,$ with the exact value $W_1\left(\mu^{(\Phi,\p)},\mu^{(\Phi,\q)} \right)=\frac{2}{5}.$ The results are in table \ref{tab:table3}.

{\tiny
\begin{table}[h!]
  \begin{center}
    \caption{Approximate values for the Wasserstein distances.}
    \label{tab:table3}
    \begin{tabular}{l  | c|r} 
      \textbf{Iteration} & \textbf{Approx. Wasserstein dsitance} & \textbf{Error} \\
      $k$ & $w_k$ &    \\
      \hline
     8 &  $0.3999999999583169972032649579562794534915$ & $<10^{-11}$\\\ 
     9 &    $0.400000000000063388018611127635519504188$  & $<10^{-13}$\\\
     10 &   $0.3999999999999999530973195805847764869926$  & $<10^{-16}$ \\\
     11 &  $0.4000000000000000000170204375320820851671$  & $<10^{-19}$ \\\ 
     12 &   $0.3999999999999999999999969551351875258863$  & $<10^{-23}$\\\
     13 &  $0.4000000000000000000000000002694598038564$ & $<10^{-27}$\\\
     14 &  $0.3999999999999999999999999999999881752292$ & $<10^{-31}$\\\
     15 &  $0.4000000000000000000000000000000000001067$ & $<10^{-36}$\\\
       16 &  $0.3999999999999999999999999999999999999999972$ & $<10^{-41}$         
    \end{tabular}
  \end{center}
\end{table}
}
\end{exam}

On the other hand, when the contractions in $\Phi$ are non-affine, an analytic formula is not available. 

\begin{exam}
Let $\Phi=\left\{ \frac{\sin(\pi x /4)}{6}+\frac{1}{4},\frac{\sin(\pi x /4)}{3}+\frac{2}{3}\right\}$ and $\p=\left( \frac{1}{7},\frac{6}{7} \right),$  $\q=\left( \frac{1}{2},\frac{1}{2} \right).$ We can compute the approximate values for the Wasserstein distance between the stationary probability measures associated to $(\Phi,\p)$ and $(\Phi,\q).$ The results are in Table \ref{tab:table4}. 

{\tiny
\begin{table}[h!]
  \begin{center}
    \caption{Approximate values for the Wasserstein distances.}
    \label{tab:table4}
    \begin{tabular}{l  | l} 
      \textbf{Iter.} & \textbf{Approx. Wasserstein distance} \\
      $k$ & $w_k$     \\
      \hline
     8 &  $0.2210457542228009986686646214111284400538314616500475049250670839710830921157028 $ \\\ 
     9 &    $0.2210457542228009986686646648222322435005592621138097135184119027209843678362625 $  \\\
     10 &  $0.2210457542228009986686646648222083279218347168575015785464317336794134748041328 $ \\\
     11 &  $0.2210457542228009986686646648222083279244322328476185036762204108148528502575738 $  \\\ 
     12 &  $0.2210457542228009986686646648222083279244322327918382319334332352348209033215652 $ \\\
     13 &  $0.2210457542228009986686646648222083279244322327918382321725464968058605768813675 $ \\\
     14 &  $0.2210457542228009986686646648222083279244322327918382321725464966008628827356344 $ \\\
     15 &  $0.2210457542228009986686646648222083279244322327918382321725464966008628827708159 $ \\\
    \end{tabular}
  \end{center}
\end{table}
}

\end{exam}

\subsection{Lyapunov exponents} \label{lyap}

Let $\Phi$ be an iterated function system and $\mu$ be the stationary probability measure associated to $(\Phi,\p).$ The Lyapunov exponent of $\Phi$ with respect to  $\mu$ can be defined by 
$$ 
\chi_{\mu}:=- \int \sum_{i=1}^N p_i \log |\phi_i'(x)|  \textup{d}\mu(x).
$$
The Lyapunov exponent of $\mu$ describes the typical rate at which distances are contracted under the action of the maps in the iterated function system, from the point of view of the measure $\mu$. Therefore, Lyapunov exponents play an important role in fractal geometry; for example the Hausdorff dimension of $\mu$ is given in terms of its Lyapunov exponent, which in turn can provide bounds or even precise values for the Hausdorff dimension of the attractor $\Lambda$, see \cite{Falconer}. When $\Phi$ is non-overlapping, the Lyapunov exponent of $\mu$ has an equivalent interpretation as the $\mu$-typical rate of expansion of the dynamical system whose branches are given by the inverse images of the maps in $\Phi$, and therefore Lyapunov exponents also play an important role in ergodic theory, particularly in thermodynamic formalism. We can use Theorem \ref{main} to approximate the Lyapunov exponent of $\Phi$ with respect to $\mu$. 

\begin{cor}

Let $\p=(p_1, \ldots, p_N)$ a probability vector and let $\Phi=\{\phi_i\}_{i \in \I}$ be a complex contracting iterated function system on the unit interval such that:
\begin{equation}
\max_{i\in \I} \sup_{x \in \phi_i([0,1])}- \log | \phi '_i(x) | \leq M<\infty. \label{control}
\end{equation}
Then Theorem \ref{main} provides an algorithm that gives approximations $w_k$  $(k=1,2,\ldots)$ to the  Lyapunov exponent of $\Phi$ with respect to $\mu$ which converges at the super-exponential rate (\ref{error}).
\end{cor}

\begin{proof} 
Follows by applying Theorem \ref{main} to the function $g(x)= - \sum_{i=1}^N p_i \log |\phi_i'(x)|$ which belongs to $\mathcal{C}_{\Phi}$ by (\ref{control}) and the fact that $\Phi$ is complex contracting.
\end{proof}

\begin{exam}
Let $\Phi=\{\frac{\sin(\pi x /4)}{6}+\frac{1}{4},\frac{\sin(\pi x /4)}{3}+\frac{2}{3}\}$ and $\p=\left( \frac{1}{3},\frac{2}{3} \right).$ We compute the approximate values to the Lyapunov exponent of $\Phi$ with respect to  $\mu$. The results are in Table \ref{tab:table6}.


{\tiny
\begin{table}[h!]
  \begin{center}
    \caption{Approximate values for Lyapunov exponents.}
    \label{tab:table6}
    \begin{tabular}{l  | l} 
      \textbf{Iter.} & \textbf{Approx. Lyapunov exponent} \\
      $k$ & $w_k$     \\
      \hline
     8 &    $1.73672081473731987719335668233174352405940161192454820335830004878102455942170283350510993376437455813744965215$ \\\ 
     9 &    $1.73672081473731987719335669096051704652553164845149915406190548269505610944642757530134017955544773131075328248$  \\\
     10 &  $1.73672081473731987719335669096051377335988973317409087754512488812667028438133069194736931707026244323079276022$ \\\
     11 &  $1.73672081473731987719335669096051377336020590601023577825778088227884147013275372835954447415750580402943072648$  \\\ 
     12 &  $1.73672081473731987719335669096051377336020590600637607989863098312523352089178183012079770679720814305305511767$ \\\
    13 &  $1.73672081473731987719335669096051377336020590600637607991887362479762663771464637789116844202618748250354445394$ \\\
     14 &  $1.73672081473731987719335669096051377336020590600637607991887362479193249845331374280452940254564000080656340808$ \\\
    15 &  $1.73672081473731987719335669096051377336020590600637607991887362479193249845555716884110153252308930493603697019$ \\\
      16 &  $1.73672081473731987719335669096051377336020590600637607991887362479193249845555716884109104549736936505689675419$ \\\
      17& $1.73672081473731987719335669096051377336020590600637607991887362479193249845555716884109104549736965148549578668$ \\\
      18& $1.73672081473731987719335669096051377336020590600637607991887362479193249845555716884109104549736965148549578625$ \\\
    \end{tabular}
  \end{center}
\end{table}
}
\end{exam}

\section{Further remarks}

\subsection{Implementation}

The algorithm was implemented in Python using the free library mpmath \footnote{Python library for real and complex floating-point arithmetic with arbitrary precision.}. To estimate the fixed points $z_{\i}$ with arbitrary precision the command 
\begin{verbatim} mpmath.findroot \end{verbatim} 
was used and to estimate $\phi'_{\i}(z_{\i})$ with arbitrary precision the command
\begin{verbatim} mpmath.diff \end{verbatim} 
was used. 

\subsection{Other applications}

Our methods can be used to estimate other functions in probability theory, such as pointwise estimates of the Fourier-Stieltjes transform of stationary probability measures $\hat{\mu}(x):=\int_0^1 e^{-2\pi i xy}\textup{d}\mu(y)$ and coefficients of the Fourier series of singular measures arising as stationary probability measures.

\subsection{Extension to non constant weight functions}

Let $\Phi$ be a complex contracting iterated function as before. We can extend our results to certain probability measures associated to \emph{non-constant} weight functions. 

In particular, suppose there exists an admissable domain $D$ for $\Phi$ and bounded holomorphic functions $\p=\{p_i\}_{i\in \I}$, $p_i: D \to \C$ with the property that
\begin{enumerate}
\item $\sum_{i \in \I} p_i(z)=1$ for all $z \in D$,
\item $p_i([0,1]) \subset (0,1)$ for all $i \in \mathcal{I}$.
\end{enumerate}
Then we can obtain an analogue of Theorem \ref{main} for the unique probability measure $\mu=\mu^{(\Phi,\p)}$ supported on the attractor of $\Phi$ with the property that
\begin{equation}\label{SPMeq}
\int \varphi(x) \textup{d}\mu(x) = \sum_{i\in\I}  \int p_i(x) \varphi \circ \phi_i (x) \textup{d}\mu (x)
\end{equation}
for every continuous function $\varphi:[0,1]\to \R$  . This probability measure is called the stationary probability measure associated to $\Phi$ and $\p$, see \cite{Fan_Lau_1999}. It generalises the stationary probability measure considered for constant weight functions.

By following the proof of Theorem \ref{main} with the operator
$$\l_sf(z)= \sum_{i=1}^N p_i(z) \exp(sg(\phi_i (z)))f(\phi_i(z))$$
one can obtain the following analogue of Theorem \ref{main}.

\begin{thm}\label{generalization_Theorem}
Let $\Phi=\{\phi_i\}_{i \in \I}$, $\p=\{p_i\}_{i\in\I}$ and $\mu=\mu^{\Phi, \p}$ be as above. Given $g\in \mathcal{C}_{\Phi},$ define
$$t_m:=\sum_{\i \in \I^m}  \frac{p_{\i}(z_{\i}) }{1-\phi_{\i}'(z_{\i})}$$
and
$$\tau_m:=\sum_{\i \in \I^m} p_{\i}(z_{\i}) \frac{g(z_{\i})+ g(z_{\sigma \i})+ \ldots + g(z_{\sigma^m \i}) }{1-\phi_{\i}'(z_{\i})},$$
where
$$ p_{\i}(z) := p_{i_m }(z) p_{i_{m-1} }(\phi_{i_m}z)\cdots p_{i_1}( \phi_{i_2 \ldots i_m} z ).$$
Define $\alpha_n$, $a_n$ and $\mu_k(g)$ as in Theorem \ref{main}. We have
\begin{eqnarray}
\left|\int g(x) \textup{d} \mu (x)-\mu_k(g)\right| < C\exp(-\lambda k^2)
\label{error}
\end{eqnarray} for some constants $C, \lambda>0$ which are independent of $k.$
\end{thm}

\vspace{5mm}

\noindent \textbf{Acknowledgements.}

IC would like to express his gratitude to Felipe Serrano and  Franziska Schl\"osser for suggestions around the effective implementation of the algorithms. IC was partially supported by CONICYT PIA ACT172001. NJ was financially supported by the \emph{Leverhulme Trust} (Research Project Grant number RPG-2016-194). Both authors would like to thank Ian Morris for helpful discussions and suggestions.

\end{document}